\documentclass[12pt, reqno]{amsart}
\usepackage{amsmath}  \hoffset=-56pt

\usepackage{amsmath, amsfonts, rawfonts,amssymb,txfonts, enumerate}
\usepackage{eucal}
\usepackage{latexsym}
\usepackage{cite, color}
 \usepackage{bm}

\newcommand{\rnnnn}{\mathbb R^+}
\newcommand{\rnnn}{\mathbb R^{n+1}}

\newcommand{\rt}{\mathbb R^3}
\newcommand{\sn}{ {\mathbb{S}^{n}}}

\newcommand{\psum}{{+_{\negthinspace\kern-2pt p}}\,}
\newcommand{\qsum}[1]{{+_{\negthinspace\kern-2pt #1}}\,}
\newcommand{\dpsum}{{\tilde+_{\negthinspace\kern-1pt p}}\,}
\newcommand{\dqsum}[1]{{\tilde+_{\negthinspace\kern-1pt #1}}\,}
\newcommand{\lsub}[1]{\hskip -1.5pt\lower.5ex\hbox{$_{#1}$}}

\numberwithin{equation}{section}

\newtheorem{theo}{Theorem}[section]
\newtheorem{coro}[theo]{Corollary}
\newtheorem{lem}[theo]{Lemma}
\newtheorem{prop}[theo]{Proposition}

 \theoremstyle{definition}

\begin{document}

\title{Pinching estimates of hypersurfaces by a generalized Gauss curvature flow}

\author[J. Hu]{Jingrong Hu}
\address{School of Mathematics, Hunan University, Changsha, 410082, Hunan Province, China}
\email{hujinrong@hnu.edu.cn}
\author[P. Zhang]{Ping Zhang}
\address{School of Mathematics, Hunan University, Changsha, 410082, Hunan Province, China}
\email{zping071727@hnu.edu.cn}

\begin{abstract}
A variant of the Gauss curvature flow for closed and convex hypersurfaces is considered.  We reveal that if the initial hypersurface is pinched enough, then this property is preserved. Furthermore, based on some structure assumptions on the speed function of the shrinking flow, we show that the flow converges to a sphere. This may generalize the result of B. Chow\cite{CW85} to the possible non-homogeneous curvature flows.
\end{abstract}

\maketitle

\baselineskip18pt

\parskip3pt

\section{Introduction}

In this paper, we investigate a
generalized Gauss curvature flow. More precisely, given an integer $n\geq 2$,  let $F_{0}:\sn \rightarrow M_{0}\subset \rnnn(n\geq 2)$ be a parametrization of a smooth, strictly convex hypersurface $M_{0}$. The motion of the convex hypersurfaces $M_{t}$ satisfy the curvature flow equation:
\begin{equation}\label{ca}
\left\{
\begin{array}{lr}
\frac{\partial F}{\partial t}(x,t)=-f(K(x,t))v(x,t), \\
F(x,0)=\vartheta F_{0}(x), \quad x\in \sn.
\end{array}\right.
\end{equation}
Here $\vartheta>0$,  $F(\cdot, t)$ parametrizes $M_{t}$, $v(x,t)$ is the unit outer normal vector at $F(x,t)$, $K$ is the Gauss curvature of $M_{t}$, and $f:\rnnnn\rightarrow \rnnnn$ is a smooth function.

In particular, $\alpha$-Gauss curvature flow ($\alpha$-GCF in abbreviation) have been studied by many authors in the sense that $f(K)=K^{\alpha}$ with $\alpha>0$. The classical Gauss curvature flow (GCF), $\alpha=1$ case, was first introduced by Firey \cite{F74} to describe a model for the wearing of tumbling stones in $\rt$. Later, Tso \cite{T85} revealed that GCF exists up to some maximal time in the circumstance that the enclosed volume converges to zero in $\rnnn$.  In the $\alpha=1/(n+2)$, Andrews\cite{An96} obtained the convergence of the flow to an ellipsoid. For the range of $\alpha>1/(n+2)$, Chow\cite{CW85} showed that the flow converges to a round sphere for $\alpha=1/n$. Andrews-Guan-Ni\cite{Ag16} proved the convergence to a self-similar solition for $\alpha>1/(n+2)$. In \cite{BCD17}, Brendle-Choi-Daskalopoulos showed that a closed strictly convex solution of the $\alpha$-GCF converges to a round sphere in $\rnnn$ for all $\alpha > 1/(n+2)$.

Some general results for flowing surfaces by using nonhomogeneous speeds (other than powers of Gauss curvature) were obtained, for instance, Chow and Tsai \cite{CW98} studied the expansion of a smooth closed convex hypersurface by a nonhomogeneous function of the reciprocal of the Gauss curvature, proved that the hypersurface expand to infinity with its shape becoming round asymptotically starting from any smooth and strictly convex hypersurface,  and got some partial results for contraction and bidirectional flows. In \cite{CW00}, Chou-Wang concerned a sort of logarithmic version of the Gauss curvature flow for  convex hypersurfaces. To the best of our knowledge, the asymptotical behavior for flow of convex hypersurfaces by arbitrary speeds of the Gauss curvature have been relatively much less well investigated. Subjecting to this topic, we are devoted to studying the convergence of the shrinking flow showed in \eqref{ca}. For this purpose, we require that the speed function $f(K)$ satisfies the following condition:

{\bf Condition 1.1.}
\begin{itemize}

    \item [(i)] $f^{'}(K)> 0$;
    \item [(ii)] $\alpha_{1}\leq\frac{nKf^{'}(K)-f(K)}{f(K)}\leq \alpha_{2}$ for positive constants $0<\alpha_{1}< \alpha_{2}\leq n-1$;
        \item [(iii)] $0< \frac{K f^{''}(K)}{f^{'}(K)}+1-\frac{1}{n}\leq \beta $ for a positive constant $\beta$;
            \item [(iv)]  $\frac{K}{f(K)^{\gamma}}\geq \hat{\gamma}$ for sufficiently large $K$, and positive constants $\gamma, \hat{\gamma}$.
                \item [(v)] $f(K)$ is convex with respect to the radii of principal curvature of convex hypersurface $M_{t}$.
    \end{itemize}

 We notice that Condition (i) ensures that \eqref{ca} is locally a parabolic system. Short-time existence and uniqueness of solutions can be obtained via standard argument (see \cite{Ge06}).  Different from the results obtained by \cite{CW98}, our study sheds light on the pinching estimates of hypersurfaces by a contraction flow in the nonhomogeneous Gauss curvature form. We along with the similar lines as presented in Chow's work . In \cite{CW85}, Chow made use of the quantity $K/H^{n}$  and showed that convex hypersurface evolving under the $n$-th root of the Gauss-curvature converge shrinks to a sphere after suitable rescaling. He also revealed that a similar result for normal velocities whose degree of homogeneity is greater than 1. After then, the results of Chow for Gauss-curvature were extended, such as in a series of paper of Andrews \cite{An94, An07, An10} with speeds given by homogeneous of degree 1 functions of the principal curvatures.  However, shrink flows by non-homogeneous Gauss curvature flows may have received relatively little attention. Heuristically, we give the following result.
\begin{theo}\label{th*}
Let $F(\cdot,t)$ be a family of hypersurfaces evolving under the curvature flow \eqref{ca}. Under the assumption of Condition 1.1, if the initial hypersurface is pinched in the sense that
\begin{equation}\label{conv&}
\frac{K(x,t)}{H^{n}(x,t)}\geq C(n ), \quad \forall x\in \sn
\end{equation}
for a positive constant $C(n)< 1/n^{n}$, then \eqref{conv&} is preserved under the flow. Furthermore, there exist a $\vartheta^{*}>0$ such  that the family of pinched hypersurface beginning at $\vartheta^{*}F_{0}(x)$ converges to a sphere.
\end{theo}

The main contribution of current paper is that we can remove the homogeneity assumption on $f$ for shrinking flows and still get results similar to the homogeneous case in \cite{CW85}. Note that analogous results showed in Theorem \ref{th*} have been obtained by a various of other flows, for instance, Huisken showed that convex hypersurfaces evolving by mean curvature contract to a point in finite time, and become spherical in shape as the limit is achieved in his seminal paper \cite{H84}. After that, the power of mean curvature flow was considered by Schulze \cite{S06}. Very recently, Espin \cite{Es22} employed a non-homogeneous variant of mean curvature flow.

The organization of this paper goes as follows. In Sect. \ref{Sec2}, we give some basics on hypersurface. In Sect. \ref{Sec3}, some necessary evolution equations are listed. In Sect. \ref{Sec4}, we show that if the initial hypersurface is pinched strongly enough, then this preserved under the flow. In the last, we reveal that the flow converges to a sphere, and give  non-homogeneous examples on $f(K)$ responding to Theorem \ref{th*}.

\section{Notations and Preliminaries}
\label{Sec2}
 In this section, we list some facts on the convex hypersurfaces that shall be used in the subsequent contents.

Let $M_{t}\subset \rnnn$ be a convex hypersurface, the second fundamental form $h_{ij}$ of $M_{t}$ is given by $-\langle \partial^{2}F/ \partial x_{i}\partial x_{j},v\rangle$, the induced metric $g_{ij}$  on $M_{t}$ is represented by $\langle \partial F/\partial x_{i}, \partial F/ \partial x_{j}\rangle$ in the ordinate system $\{x_{i}\}$, and the Gauss curvature of $M_{t}$ is given by $K=\det h_{ij} /  \det g_{ij}$.  If $\zeta$ and $\xi$ are scalars we use the notation $\langle \nabla \zeta, \nabla \xi \rangle:=g^{ab}\nabla_{a}\zeta \nabla_{b}\xi$ and $|\nabla \xi|^{2}=\langle \nabla \xi, \nabla \xi \rangle$. If $X$ and $Y$ are 1-forms and  $s$ is a symmetric positive definite covariant 2-tensor on $M_{t}$, then define
$ \langle X, Y\rangle_{s}=s^{-1}_{ij}X_{i}X_{y}$. If $t$ is a contravariant 2-tensor, define $\langle X, Y\rangle_{t^{-1}}=t^{ij}X_{i}X_{j}$. If $A_{ikl}$ is a covariant 3-tensor, then
\[
|A_{ikl}|^{2}_{g,h}=g^{ij}h^{-1}_{km}h^{-1}_{ln}A_{ikl}A_{jmn}.
\]
We are in the place to recall the following fundamental formulas of a hypersurface $\rnnn$,
\begin{equation*}
\nabla_{i}h_{kl}=\nabla_{k}h_{il}=\nabla_{l}h_{ik} \quad (Codazzi\ equation),
\end{equation*}
\begin{equation*}
R_{ijkl}=h_{ik}h_{jl}-h_{il}h_{jk}\quad (Gauss \ equation),
\end{equation*}
where $R_{ijkl}$ is the curvature tensor, and we obtain the following formulas:
\begin{equation*}\label{}
h_{ijkl}=h_{ijlk}+h_{mj}R_{imlk}+h_{im}R_{jmlk},
\end{equation*}
\begin{equation*}\label{}
h_{ijkl}=h_{klij}+(h_{mj}h_{il}-h_{ml}h_{ij})h_{mk}+(h_{mj}h_{kl}-h_{ml}h_{kj})h_{mi}.
\end{equation*}

\section{Evolution equations}
\label{Sec3}
In this section, we give the evolution equations of some quantities.

\begin{prop}\label{pa}
Let $F(\cdot,t)$ satisfy the flow equation \eqref{ca}, then we obtain the following evolution equations:
\begin{equation}\label{a1}
\frac{\partial } {\partial t}g_{ij}=-2f(K)h_{ij},
\end{equation}
\begin{equation}\label{a2}
\frac{\partial } {\partial t}g^{ij} =2f(K)g^{ik}g^{jl}h_{kl},
\end{equation}
\begin{equation}\label{a3}
\frac{\partial } {\partial t}v =f^{'}(K)\nabla K,
\end{equation}
\begin{equation}\label{a4}
\frac{\partial } {\partial t}h_{ij}=f^{''}(K)\nabla_{i}K\nabla_{j}K+f^{'}(K)\nabla_{ij}K-fh_{ik}h_{kj},
\end{equation}
where $h_{ik}h_{kj}:=h_{ik}g^{km}h_{mj}$.
\end{prop}
\begin{proof}
The computation follows the similar lines as showed in \cite{H84}. Here for readers' convenience, we give the details.
For \eqref{a1},
\begin{equation}
\begin{split}
\label{b1}
\frac{\partial }{\partial t}g_{ij}&=\frac{\partial}{\partial t}\Big\langle \frac{\partial F}{\partial x_{i}},\frac{\partial F}{\partial x_{j}}\Big\rangle\\
&=\Big\langle \frac{\partial}{\partial x_{i}}(-f(K)v),\frac{\partial F}{\partial x_{j}}\Big\rangle+\Big\langle \frac{\partial F}{\partial x_{i}},\frac{\partial}{\partial x_{j}}(-f(K)v)\Big\rangle\\
&=-f(K)\Big\langle \frac{\partial }{\partial x_{i}}v,\frac{\partial F}{\partial x_{j}}\Big\rangle-f(K)\Big\langle \frac{\partial }{\partial x_{i}}F,\frac{\partial v}{\partial x_{j}}\Big\rangle\\
&=-2f(K)h_{ij}.
\end{split}
\end{equation}
Then,  \eqref{a2} immediately is obtained by using \eqref{a1}. For \eqref{a3},
\begin{equation*}
\begin{split}
\label{Up3}
\frac{\partial }{\partial t}v&=\Big\langle\frac{\partial }{\partial t}v,\frac{\partial F}{\partial x_{i}}\Big\rangle\frac{\partial F}{\partial x_{j}}g^{ij}\\
&=-\Big\langle v,\frac{\partial}{\partial t}\frac{\partial F}{\partial x_{i}}\Big\rangle\frac{\partial F}{\partial x_{j}}g^{ij}\\
&=f^{'}(K)\frac{\partial}{\partial x_{i}}K\cdot \frac{\partial F}{\partial x_{j}}g^{ij}\\
&=f^{'}(K)\nabla K.
\end{split}
\end{equation*}
For \eqref{a4},
\begin{equation*}
\begin{split}
\label{Up3}
\frac{\partial}{\partial t}h_{ij}&=-\frac{\partial }{\partial t}\Big\langle \frac{\partial^{2}F}{\partial x_{i}\partial x_{j}},v\Big\rangle\\
&=\Big\langle\frac{\partial^{2}}{\partial x_{i}\partial x_{j}}(f(K)v),v\Big\rangle-\Big\langle\frac{\partial^{2}F}{\partial x_{i}\partial x_{j}}, f^{'}(K)\frac{\partial }{\partial x_{l}}K\cdot \frac{\partial F}{\partial x_{m}}g^{lm}\Big\rangle\\
&=f^{''}(K)\frac{\partial }{\partial x_{i}}K\frac{\partial }{\partial x_{j}}K+f^{'}(K)\frac{\partial^{2}K}{\partial x_{i}\partial x_{j}}+f(K)\Big\langle \frac{\partial }{\partial x_{i}}(h_{jm}g^{ml}\frac{\partial F}{\partial x_{l}}),v\Big\rangle \\
&\quad -\Big\langle \Gamma^{k}_{ij}\frac{\partial F}{\partial x_{k}}-h_{ij}v,f^{'}(K)\frac{\partial }{\partial x_{l}}K\cdot \frac{ \partial F}{\partial x_{m}}g^{lm}\Big\rangle\\
&=f^{''}(K)\frac{\partial }{\partial x_{i}}K\frac{\partial }{\partial x_{j}}K+f^{'}(K)\frac{\partial^{2}K}{\partial x_{i}\partial x_{j}}-f^{'}(K)\Gamma^{k}_{ij}\frac{\partial }{\partial x_{k}}K+ f(K)h_{jm}g^{ml}\Big\langle \Gamma^{\sigma}_{il}\frac{\partial F}{\partial x_{\sigma}}-h_{il}v,v\Big\rangle\\
&=f^{''}(K)\nabla_{i}K\nabla_{j}K+f^{'}(K)\nabla_{ij}K-f(K)h_{ik}h_{kj}.
\end{split}
\end{equation*}

\end{proof}

\begin{lem}
Under the flow equation \eqref{ca}, we have
\begin{equation}
\begin{split}
\label{hij1}
\frac{\partial }{\partial t}h_{ij}&=f^{''}(K)\nabla_{i}K\nabla_{j}K-fh_{ik}h_{kj}\\
&\quad +f^{'}(K)\left(K\Lambda h_{ij}+HKh_{ij}-nh_{ik}h_{kj}K+\frac{\nabla_{i}K\nabla_{j}K}{K}+K\nabla_{i}h^{-1}_{kl}\nabla_{j}h_{kl}\right),
\end{split}
\end{equation}
or equivalently
\begin{equation}
\begin{split}
\label{hij2}
\frac{\partial }{\partial t}h_{ij}=&f^{'}(K)K\Lambda h_{ij}+f^{'}(K)HKh_{ij}-(f+nKf^{'}(K))h_{ik}h_{kj}+\left(f^{''}(K)+\frac{f^{'}}{K}-\frac{f^{'}}{nK}\right)\nabla_{i}K\nabla_{j}K\\
&\quad +Kf^{'}(K)\left[ -\frac{1}{H^{2}}h^{-1}_{km}h^{-1}_{lm}(H\nabla_{i}h_{mn}-\nabla_{i}H h_{mn})(H\nabla_{j}h_{kl}-\nabla_{j}Hh_{kl})+\frac{H^{2n}}{nK^{2}}\nabla_{i}\left( \frac{K}{H^{n}}\right)\nabla_{j}\left(\frac{K}{H^{n}}\right)\right],
\end{split}
\end{equation}
where $\Lambda=h^{-1}_{kl}\nabla_{k}\nabla_{l}$.
\end{lem}
\begin{proof}
Since
$K=\det h_{ij}/ \det g_{ij}$, then
\begin{equation}\label{k1}
\nabla_{j}K=Kh^{-1}_{kl}\nabla_{j}h_{kl}.
\end{equation}
This together with Gauss equation, we have
\begin{equation}
\begin{split}
\label{hij}
\nabla_{ij}K&=Kh^{-1}_{kl}\nabla_{ij}h_{kl}+Kh^{-1}_{mn}\nabla_{i}h_{mn}h^{-1}_{kl}\nabla_{j}h_{kl}+K\nabla_{i}h^{-1}_{kl}\nabla_{j}h_{kl}\\
&=Kh^{-1}_{kl}(\nabla_{kl}h_{ij}+(h_{ij}h_{km}-h_{im}h_{kj})h_{ml}+(h_{il}h_{km}-h_{im}h_{kl})h_{mj})\\
&\quad +\frac{\nabla_{i}K\nabla_{j}K}{K}+K\nabla_{i}h^{-1}_{kl}\nabla_{j}h_{kl}\\
&=K\Lambda h_{ij}+HKh_{ij}-nh_{ik}h_{kj}K+\frac{\nabla_{i}K\nabla_{j}K}{K}+K\nabla_{i}h^{-1}_{kl}\nabla_{j}h_{kl}.
\end{split}
\end{equation}
Substituting \eqref{hij} into (2.4), we get
\begin{equation}
\begin{split}
\label{FK}
\frac{\partial }{\partial t}h_{ij}&=f^{'}(K)K\Lambda h_{ij}+f^{'}(K)HKh_{ij}-(f+nKf^{'}(K))h_{ik}h_{kj}+\left(f^{''}(K)+\frac{f^{'}(K)}{K}\right)\nabla_{i}K\nabla_{j}K\\
&\quad +Kf^{'}(K)\nabla_{i}h^{-1}_{kl}\nabla_{j}h_{kl}.
\end{split}
\end{equation}
Since
\begin{equation*}
\begin{split}
\label{Up3}
\frac{1}{nK^{2}}\nabla_{i}K\nabla_{j}K+\nabla_{i}h^{-1}_{kl}\nabla_{j}h_{kl}&=\frac{H^{2n}}{nK^{2}}\nabla_{i}\left(\frac{K}{H^{n}}\right)\nabla_{j}\left(\frac{K}{H^{n}}\right)-\frac{n}{H^{2}}\nabla_{i}H\nabla_{j}H\\
&\quad + \frac{1}{HK}(\nabla_{i}H\nabla_{j}K+\nabla_{j}K\nabla_{i}H)-h^{-1}_{km}h^{-1}_{ln}\nabla_{i}h_{mn}\nabla_{j}h_{kl},
\end{split}
\end{equation*}
and
\begin{equation*}
\begin{split}
\label{Up3}
&\frac{1}{H^{2}}h^{-1}_{km}h^{-1}_{ln}(H\nabla_{i}h_{mn}-\nabla_{i}Hh_{mn})(H\nabla_{j}h_{kl}-\nabla_{j}Hh_{kl})\\
&=h^{-1}_{km}h^{-1}_{ln}\nabla_{i}h_{mn}\nabla_{j}h_{kl}-\frac{1}{HK}(\nabla_{i}H\nabla_{j}K+\nabla_{i}K\nabla_{j}H)+\frac{n}{H^{2}}\nabla_{i}H \nabla_{j}H.
\end{split}
\end{equation*}
This implies
\begin{equation}
\begin{split}
\label{kk}
&\frac{1}{nK^{2}}\nabla_{i}K \nabla_{j}K+\nabla_{i}h^{-1}_{kl}\nabla_{j}h_{kl}\\
&=-\frac{1}{H^{2}}h^{-1}_{km}h^{-1}_{lm}(H\nabla_{i}h_{mn}-\nabla_{i}Hh_{mn})(H\nabla_{j}h_{kl}-\nabla_{j}Hh_{kl})\\
&\quad +\frac{H^{2n}}{nK^{2}}\nabla_{i}\left(\frac{K}{H^{n}}\right)\nabla_{j}\left(\frac{K}{H^{n}}\right).
\end{split}
\end{equation}
Then, substituting \eqref{kk} into \eqref{FK}, we have the desired result.
\end{proof}
\begin{lem}\label{KHTT}
Under the flow equation \eqref{ca}, we have
\begin{equation}
\begin{split}
\label{htt}
\frac{\partial }{\partial t}H&=f^{'}(K)K\Lambda H+f^{'}K H^{2}+\frac{H^{2n}f^{'}(K)}{nK}\Big|\nabla\left( \frac{K}{H^{n}}\right)\Big|^{2}+(f-nKf^{'}(K))|A|^{2}\\
&\quad +\left(f^{''}(K)+\frac{f^{'}(K)}{K}-\frac{f^{'}(K)}{nK}\right)|\nabla K|^{2}-\frac{Kf^{'}(K)}{H^{2}}|H\nabla_{i}h_{kl}-\nabla_{i}Hh_{kl}|^{2}_{g,h},
\end{split}
\end{equation}
and
\begin{equation}
\begin{split}
\label{Kt}
\frac{\partial }{\partial t}K=Kf^{'}(K)\Lambda K+Kf^{''}(K)|\nabla K|^{2}_{h}+fKH.
\end{split}
\end{equation}
\end{lem}
\begin{proof}
Applying \eqref{a2} and \eqref{hij1}, we have
\begin{equation}
\begin{split}
\label{HT}
\frac{\partial }{\partial t}H&=\frac{\partial}{\partial t}(g^{ij}h_{ij})\\
&=f^{'}(K)K\Lambda H+f^{'}(K)KH^{2}-(f+nKf^{'}(K))|A|^{2}+\left(f^{''}(K)+\frac{f^{'}(K)}{K}-\frac{f^{'}(K)}{nK^{2}}\right)|\nabla K|^{2}\\
&\quad -\frac{Kf^{'}(K)}{H^{2}}|H\nabla_{i}h_{kl}-\nabla_{i}Hh_{kl}|^{2}_{g,h}+\frac{H^{2n}Kf^{'}(K)}{nK^{2}}\Big|\nabla\left( \frac{K}{H^{n}}\right)\Big|^{2}-g^{im}g^{nj}\partial_{t}g_{mn}h_{ij}\\
&=f^{'}(K)K\Lambda H+f^{'}K H^{2}+\frac{H^{2n}f^{'}(K)}{nK}\Big|\nabla\left( \frac{K}{H^{n}}\right)\Big|^{2}+(f-nKf^{'}(K))|A|^{2}\\
&\quad +\left(f^{''}+\frac{f^{'}}{K}-\frac{f^{'}}{nK}\right)|\nabla K|^{2}-\frac{Kf^{'}(K)}{H^{2}}|H\nabla_{i}h_{kl}-\nabla_{i}Hh_{kl}|^{2}_{g,h}.
\end{split}
\end{equation}
On the other hand, by the definition of $K$, one has
\begin{equation}
\begin{split}
\label{KT}
\frac{\partial }{\partial t}K&=\frac{\det h_{ij}}{\det g_{ij}}(h^{-1}_{ij}\partial_{t}h_{ij}-g^{ij}\partial_{t}g_{ij})\\
&=K h^{-1}_{ij}(f^{''}(K)\nabla_{i}K\nabla_{j}K+f^{'}(K)\nabla_{ij}K-fh_{ik}h_{kj})+2f(K)Kg^{ij}h_{ij}\\
&=Kf^{'}(K)\Lambda K+Kf^{''}(K)|\nabla K|^{2}_{h}+fKH.
\end{split}
\end{equation}
The proof is completed.
\end{proof}

\section{Pinching estimates}

\label{Sec4}

In this section, we mainly prove the following property.
\begin{prop}\label{pink}
 There exists a constant $0<C(n)<1/n^{n}$ such that if the initial hypersurface $M_{0}$ pinched in the sense that
\begin{equation}\label{pin}
\frac{K(x,t)}{H^{n}(x,t)}\geq C(n), \quad \forall x\in  \sn,
\end{equation}
then the pinching \eqref{pin} is preserved under the flow \eqref{ca}.
\end{prop}
 Before proving Proposition \ref{pink}, we do some preparation. Using the maximum principle to \eqref{Kt} of Lemma \ref{KHTT}, we first get the following result.
\begin{coro}\label{KPR}
Under the flow \eqref{ca}. Suppose $C$ is a constant and $K\geq C$ at $t=0$, then $K\geq C$ for all time. As a result, if $M_{t}$ is strictly convex  at $t=0$, then it remains so.
\end{coro}

Along the similar lines as revealed in \cite[Lemma 4.4]{Es22}, we have
\begin{lem}\label{GH}
There exists a constant $0<C<1/n^{n}$ such that if a convex hypersurface $\mathcal{N}$ satisfies \eqref{pin} with this constant $C$, then there exists $0<\varepsilon=\varepsilon(C)\leq 1/n$ such that
\begin{equation*}\label{}
h_{ij}\geq \varepsilon H g_{ij}
\end{equation*}
and
\begin{equation*}\label{}
\Big|h^{-1}_{ij}-\frac{n}{H}g^{ij}\Big|\leq \frac{\varepsilon^{2}}{4n\beta H}
\end{equation*}
on $\mathcal{N}$. Furthermore, $\varepsilon\rightarrow 1/n$ as $C\rightarrow 1/n^{n}$.
\end{lem}

Now, we give the evolution equation of $K/H^{n}$.
\begin{lem}Under the flow equation \eqref{ca}, we have
\begin{equation}
\begin{split}
\label{KHN}
&\frac{\partial}{\partial t}\left(\frac{K}{H^{n}}\right)=Kf^{'}(K)\left[\Lambda \left(\frac{K}{H^{n}}\right)+\frac{(1/n-1)}{K}\Big\langle\nabla K,\nabla\left(\frac{K}{H^{n}}\right)\Big\rangle_{h}+\frac{n+1}{H}\Big\langle \nabla H,  \nabla \left( \frac{K}{H^{n}}\right)\Big\rangle_{h}\right]\\
&\quad +\frac{nK}{H^{n+1}}(-f+nK f^{'}(K))\left(|A|^{2}-\frac{1}{n}H^{2}\right)+\frac{1}{H^{n}}\left(Kf^{''}+f^{'}-\frac{f^{'}}{n}\right)|\nabla K|^{2}_{e}\\
&\quad +\frac{nK^{2}f^{'}(K)}{H^{n+3}}|H\nabla_{i}h_{kl}-\nabla_{i}Hh_{kl}|^{2}_{g,h}-H^{n-1}f^{'}(K)\Big|\nabla \left(\frac{K}{H^{n}}\right)\Big|^{2},
\end{split}
\end{equation}
where $e=(h^{-1}_{ij}-n/H\cdot g^{-1}_{ij})^{-1}$.
\end{lem}
\begin{proof}
Clearly,
\begin{equation*}
\begin{split}
\label{Up3}
&\frac{\partial}{\partial t}\left(\frac{K}{H^{n}}\right)=\frac{1}{H^{n}}\partial_{t}K-n  H^{-(n+1)}K\partial_{t}H\\
&=\frac{1}{H^{n}}(Kf^{'}(K)\Lambda K+Kf^{''}(K)|\nabla K|^{2}_{h}+fKH)\\
& \quad -nKH^{-(n+1)}\left[ f^{'}(K)K\Lambda H+ f^{'}(K)KH^{2}+\frac{H^{2n}f^{'}(K)}{nK}\Big| \nabla \left(\frac{K}{H^{n}}\right) \Big|^{2}+(f-nKf^{'}(K))|A|^{2}
\right.\\
&\left. \quad \quad \quad \quad \quad \quad \quad + \left(f^{''}(K)+\frac{f^{'}(K)}{K}-\frac{f^{'}(K)}{nK}\right)|\nabla K|^{2}-\frac{Kf^{'}(K)}{H^{2}}|H\nabla_{i}h_{kl}-\nabla_{i}Hh_{kl}|^{2}_{g,h}\right].
\end{split}
\end{equation*}
Since
\begin{equation*}
\begin{split}
\label{Up3}
\Lambda \left(\frac{K}{H^{n}} \right)=\frac{\Lambda K}{H^{n}}-2n\frac{\langle \nabla H, \nabla K\rangle_{h}}{H^{n+1}}+\frac{n(n+1)K|\nabla H|^{2}_{h}}{H^{n+2}}-\frac{nK\Lambda H}{H^{n+1}}.
\end{split}
\end{equation*}
Then, we obtain
\begin{equation*}
\begin{split}
\label{Up3}
\frac{\partial}{\partial t}\left(\frac{K}{H^{n}}\right)&=Kf^{'}(K)\left[\Lambda \left( \frac{K}{H^{n}}\right)+\frac{2n\langle \nabla H, \nabla K\rangle_{h}}{H^{n+1}}-\frac{n(n+1)K|\nabla H|^{2}_{h}}{H^{n+2}}\right]\\
&\quad +\frac{Kf^{''}(K)|\nabla K|^{2}_{h}}{H^{n}}+\frac{f(K)KH}{H^{n}}-nKH^{-(n+1)}f^{'}(K)KH^{2}-nKH^{-(n+1)}\frac{H^{2n}f^{'}(K)}{nK}\Big| \nabla \left( \frac{K}{H^{n}}\right)\Big|^{2}\\
&\quad -nKH^{-(n+1)}(f-nKf^{'}(K))|A|^{2}-nKH^{-(n+1)}\left( f^{''}(K)+\frac{f^{'}(K)} {K}-\frac{f^{'}(K)}{nK}\right)|\nabla K|^{2} \\
&\quad +nK H^{-(n+1)}\frac{Kf^{'}(K)}{H^{2}}|H\nabla_{i}h_{kl}-\nabla_{i}Hh_{kl}|^{2}_{g,h}.
\end{split}
\end{equation*}
Due to
\begin{equation*}\label{}
\Big\langle \nabla K, \nabla\left(\frac{K}{H^{n}} \right)\Big\rangle_{h}=\frac{\langle\nabla K, \nabla K\rangle_{h}}{H^{n}}-\frac{ nK}{H^{n+1}}\langle \nabla K, \nabla H\rangle_{h},
\end{equation*}
and
\begin{equation*}
\begin{split}
\label{Up3}
&\frac{(1/n-1)}{K}\Big\langle \nabla K,  \nabla \left( \frac{K}{H^{n}}\right)\Big\rangle_{h}+\frac{n+1}{H}\Big\langle \nabla H,  \nabla \left( \frac{K}{H^{n}}\right)\Big\rangle_{h}\\
&=\frac{1/n-1}{H^{n}K}|\nabla K|^{2}_{h}+\frac{2n}{H^{n+1}}\Big\langle \nabla K,  \nabla H\Big\rangle_{h}-\frac{(n+1)nK}{H^{n+2}}|\nabla H|^{2}_{h}.
\end{split}
\end{equation*}
So
\begin{equation}
\begin{split}
\label{para}
\frac{\partial}{\partial t}\left(\frac{K}{H^{n}}\right)&=Kf^{'}(K)\left[ \Lambda \left(\frac{K}{H^{n}}\right)+\frac{1/n-1}{K}\Big\langle \nabla K,  \nabla \left( \frac{K}{H^{n}}\right)\Big\rangle_{h}+\frac{n+1}{H}\Big\langle \nabla H,  \nabla \left( \frac{K}{H^{n}}\right)\Big\rangle_{h}\right]\\
&+\left( \frac{K f^{''}(K)}{H^{n}}-\frac{(1/n-1)f^{'}(K)}{H^{n}}\right)|\nabla K|^{2}_{h}-\frac{nK}{H^{n+1}}\left(f^{''}(K)+\frac{f^{'}(K)}{K}-\frac{f^{'}(K)}{nK}\right)|\nabla K|^{2}\\
&\quad +\frac{f K}{H^{n-1}}-\frac{nK^{2}f^{'}(K)}{H^{n-1}}-\frac{nK}{H^{n+1}}(f-nK f^{'}(K))|A|^{2}+\frac{nK^{2}f^{'}(K)}{H^{n+3}}|H\nabla_{i}h_{kl}-\nabla_{i}Hh_{kl}|^{2}_{g,h}\\
&\quad -H^{n-1}f^{'}(K)\Big| \nabla \left( \frac{K}{H^{n}}\right)\Big|^{2}\\
&=Kf^{'}(K)\left[ \Lambda \left(\frac{K}{H^{n}}\right)+\frac{1/n-1}{K}\Big\langle \nabla K,  \nabla \left( \frac{K}{H^{n}}\right)\Big\rangle_{h}+\frac{n+1}{H}\Big\langle \nabla H,  \nabla \left( \frac{K}{H^{n}}\right)\Big\rangle_{h}\right]\\
&\quad +\frac{1}{H^{n}}\left( Kf^{''}(K)+f^{'}(K)-\frac{f^{'}(K)}{n}\right)|\nabla K|^{2}_{e}+\frac{nK}{H^{n+1}}(-f+nKf^{'}(K))\left(|A|^{2}-\frac{1}{n}H^{2}\right)\\
&\quad +\frac{nK^{2}f^{'}(K)}{H^{n+3}}|H\nabla_{i}h_{kl}-\nabla_{i}Hh_{kl}|^{2}_{g,h} -H^{n-1}f^{'}(K)\Big| \nabla \left( \frac{K}{H^{n}}\right)\Big|^{2}.
\end{split}
\end{equation}
\end{proof}

{\bf Proof of Proposition \ref{pink}.} From Lemma \ref{GH}, there exists $\varepsilon>0$ satisfying
\[
h_{ij}\geq  \varepsilon H g_{ij}, \quad \Big| h^{-1}_{ij}-\frac{n}{H} g^{ij} \Big|\leq \frac{\varepsilon^{2}}{4 n \beta H}.
\]
On the one hand, applying \cite[Lemma 2.3(ii)]{H84}, we obtain the estimate
\begin{equation*}\label{}
|H\nabla_{i}h_{kl}-\nabla_{i}Hh_{kl}|^{2}\geq \frac{1}{2}\varepsilon^{2}H^{2}|\nabla H|^{2}
\end{equation*}
provided $h_{ij}\geq \varepsilon H g_{ij}$. Furthermore, due to $h_{ij}\leq H g_{ij}$, it suffices to have
\begin{equation}
\begin{split}
\label{mo1}
\frac{n K^{2}}{H^{n+3}}|H\nabla_{i}h_{kl}-\nabla_{i}Hh_{kl}|^{2}_{g,h}& \geq \frac{n K^{2}}{H^{n+3}} \frac{|H\nabla_{i}h_{kl}-\nabla_{i}Hh_{kl}|^{2}}{H^{2}}\\
&\geq \frac{n \varepsilon^{2}}{2}\frac{K^{2}}{H^{n+3}}|\nabla H|^{2}.
\end{split}
\end{equation}
On the other hand,
\begin{equation}
\begin{split}
\label{mo2}
\Bigg| \frac{1}{H^{n}}\left(\frac{Kf^{''}(K)}{f^{'}(K)} +1-\frac{1}{n}\right)|\nabla K|^{2}_{e}\Bigg|&\leq \frac{1}{H^{n}} \left(\frac{Kf^{''}(K)}{f^{'}(K)} +1-\frac{1}{n}\right)|\nabla K|^{2}\Big|h^{-1}_{ij}-\frac{n}{H}g^{ij}\Big|\\
&\leq \frac{K^{2}}{H^{n+3}}\beta\frac{n^{2}\varepsilon^{2}}{4n \beta}|\nabla H|^{2}+terms \ linear \ in \nabla\left(  \frac{K}{H^{n}}\right)\\
&= \frac{n\varepsilon^{2}}{4 }\frac{K^{2}}{H^{n+3}}|\nabla H|^{2}+terms \ linear \ in \ \nabla\left(  \frac{K}{H^{n}}\right).
\end{split}
\end{equation}
 In view of our choice $f^{'}(K)>0$ as showed in Condition (i), then \eqref{mo1} and \eqref{mo2} illustrate that
\begin{equation}
\begin{split}
\label{ji}
&
\frac{nK^{2}f^{'}(K)}{H^{n+3}}|H\nabla_{i}h_{kl}-\nabla_{i}Hh_{kl}|^{2}_{g,h} +\frac{1}{H^{n}}\left( Kf^{''}(K)+f^{'}(K)-\frac{f^{'}(K)}{n}\right)|\nabla K|^{2}_{e}\\
&\geq f^{'}(K)\frac{n\varepsilon^{2}}{4 }\frac{K^{2}}{H^{n+3}}|\nabla H|^{2}- modulo \ \nabla \left(  \frac{K}{H^{n}}\right).
\end{split}
\end{equation}
From the choice of $f$, there satisfies $nK f^{'}(K)-f>0$ by Condition (ii), and the identity $H^{2}\leq n |A|^{2}$ holds for any hypersurfaces with the aid of the Cauchy-Schwartz inequality, then we have
\begin{equation}\label{jii}
(-f+nKf^{'}(K))\left(|A|^{2}-\frac{1}{n}|H|^{2}\right)> 0.
\end{equation}
Now, we substitute \eqref{ji} and \eqref{jii} into \eqref{para}, and apply the parabolic maximum principle on \eqref{para}, we obtain the desired result.

\section{Convergence of a sphere}
\label{Sec5 }

In this section, we want to show that the shape of $M_{t}$ approaches that of a round sphere as $t\rightarrow T$. To realize this, define the auxiliary functions
\begin{equation}\label{gd}
g=\frac{1}{n^{n}}-\frac{K}{H^{n}},\quad g_{\sigma}=\phi(H)g,
\end{equation}
where $\phi(H)=H^{\sigma}$ with $\sigma>0$. Notice that $0\leq g\leq 1/n^{n}$ and $g=0$ holds at umbilic points.
\begin{theo}\label{conv} Suppose the initial hypersurface $M_{0}$ satisfy the pinching condition \eqref{conv}, and $g_{\sigma}$ is showed in \eqref{gd}, then there exists $\sigma>0$ such that
\begin{equation}\label{gxt}
g_{\sigma}(x,t)\leq \max g_{\sigma}(x,0),\quad \forall (x,t)\in \sn \times [0, T).
\end{equation}
\end{theo}
Before proving Theorem \ref{conv}, we first need the following result proved in \cite[Lemma 5.5]{CW85} or \cite[Lemma 2.5]{S06}, which holds for any convex hypersurface with $h_{ij}\geq \varepsilon H g_{ij}$.
\begin{lem}\label{AH}
Assume that $\lambda_{i}\geq \varepsilon H>0$ for some $\varepsilon >0$ and for all $i=1,\cdots n$, then there exists $\delta>0$ such that
\[
\frac{n|A|^{2}-H^{2}}{H^{2}}\geq \delta \left( \frac{1}{n^{n}}-\frac{K}{H^{n}}\right),
\]
where $\lambda_{1},\cdots, \lambda_{n}$ are the principal curvatures of convex hypersurface.

\end{lem}
Now, we show the evolution equation of $g_{\sigma}$.
\begin{lem}\label{ge}
The function $g_{\sigma}$ satisfies the following equation:
\begin{equation}
\begin{split}
\label{ggg}
\frac{\partial}{\partial t}g_{\sigma}&=Kf^{'}(K)\left[ \Lambda g_{\sigma}-2\phi^{'}\langle \nabla g,\nabla H\rangle_{h}-g\phi^{''}|\nabla H|^{2}_{h} \right] \\
&\quad+\phi f^{'}(K)(1/n-1)\left[ -H^{n}|\nabla g|^{2}_{h}+nH^{-1}K\langle \nabla H, \nabla g\rangle_{h} \right]\\
&\quad +\frac{\phi K f^{'}(K)(n+1)}{H}\Big\langle \nabla H,  \nabla g\Big\rangle_{h} -\phi\frac{nK}{H^{n+1}}(-f+nK f^{'}(K))\left(|A|^{2}-\frac{1}{n}H^{2}\right)\\
&\quad -\phi\frac{1}{H^{n}}\left(Kf^{''}+f^{'}-\frac{f^{'}}{n}\right)|\nabla K|^{2}_{e}\\
&\quad -\phi\frac{nK^{2}f^{'}(K)}{H^{n+3}}|H\nabla_{i}h_{kl}-\nabla_{i}Hh_{kl}|^{2}_{g,h}+\phi H^{n-1}f^{'}(K)\Big|\nabla g\Big|^{2}\\
&\quad  +g\phi^{'}(H)f^{'}K H^{2}+\frac{g\phi^{'}(H)H^{2n}f^{'}(K)}{nK}\Big|\nabla g\Big|^{2}+g\phi^{'}(H)(f-nKf^{'}(K))|A|^{2}\\
&\quad +g\phi^{'}(H)\left(f^{''}+\frac{f^{'}}{K}-\frac{f^{'}}{nK}\right)|\nabla K|^{2}-g\phi^{'}(H)\frac{Kf^{'}(K)}{H^{2}}|H\nabla_{i}h_{kl}-\nabla_{i}Hh_{kl}|^{2}_{g,h}.
\end{split}
\end{equation}
\end{lem}

\begin{proof}
By the definition of $g_{\sigma}$, it is direct to compute
\begin{equation}\label{11}
\nabla g_{\sigma}=\phi \nabla g +g \phi^{'} \nabla H.
\end{equation}
Then
\begin{equation*}\label{}
\nabla K=n K \frac{\nabla H}{H}-\frac{\nabla g_{\sigma}}{\phi}H^{n}+\frac{g\phi^{'}H^{n}}{\phi}\nabla H,
\end{equation*}
and
\begin{equation}\label{12}
\langle \nabla H, \nabla g\rangle=\frac{\langle \nabla g_{\sigma},\nabla H\rangle}{\phi}-g\frac{\phi^{'}}{\phi}|\nabla H|^{2},
\end{equation}
and
\begin{equation}\label{13}
|\nabla g|^{2}=\frac{|\nabla g_{\sigma}|^{2}}{\phi^{2}}-2g\frac{\phi^{'}}{\phi^{2}}\langle \nabla g_{\sigma}, \nabla H\rangle+\frac{g^{2}(\phi^{'})^{2}}{\phi^{2}}|\nabla H|^{2}.
\end{equation}
Since
\begin{equation}\label{14}
\nabla g=n H^{-(n+1)}K \nabla H-\frac{1}{H^{n}}\nabla K.
\end{equation}
This yields
\begin{equation}\label{15}
\langle \nabla K, \nabla g \rangle=- H^{n}\langle \nabla g, \nabla g\rangle+ nH^{-1}K \langle\nabla H,\nabla g\rangle.
\end{equation}
Furthermore,
\begin{equation}\label{16}
\nabla_{ij}g_{\sigma}=-\phi \left(\frac{K}{H^{n}}\right)_{ij}-2\phi^{'}(H)\left(\frac{K}{H^{n}}\right)_{i}H_{j}+g\phi^{''}(H)H_{i}H_{j}+g\phi^{'}H_{ij}.
\end{equation}
This implies that
\begin{equation}\label{17}
\Lambda g_{\sigma}=-\phi \Lambda\left(\frac{K}{H^{n}}\right)-2\phi^{'}\Big\langle \nabla \left(\frac{K}{H^{n}}\right),\nabla H \Big\rangle_{h}+g\phi^{''}(H)|\nabla H|^{2}_{h}+g\phi^{'}\Lambda H.
\end{equation}
We are in a position to have
\begin{equation}
\begin{split}
\label{gtt}
\frac{\partial}{\partial t}g_{\sigma}&=\phi \frac{\partial g}{\partial t}+g\phi^{'}(H)\frac{\partial H}{\partial t}\\
&=-\phi \frac{\partial}{\partial t}\left( \frac{K}{H^{n}}\right)+g\phi^{'}(H)\frac{\partial H}{\partial t}.
\end{split}
\end{equation}
Applying  \eqref{HT} and \eqref{KHN} into \eqref{gtt}, we have
\begin{equation*}
\begin{split}
\label{}
\frac{\partial}{\partial t}g_{\sigma}&= Kf^{'}(K)\left[ -\phi\Lambda \left(\frac{K}{H^{n}}\right)+g\phi^{'}(H)\Lambda H\right]
-\phi f^{'}(K)(1/n-1)\Big\langle\nabla K,\nabla\left(\frac{K}{H^{n}}\right)\Big\rangle_{h}\\
&\quad -\frac{\phi K f^{'}(K)(n+1)}{H}\Big\langle \nabla H,  \nabla \left( \frac{K}{H^{n}}\right)\Big\rangle_{h} -\phi\frac{nK}{H^{n+1}}(-f+nK f^{'}(K))\left(|A|^{2}-\frac{1}{n}H^{2}\right)\\
&\quad -\phi\frac{1}{H^{n}}\left(Kf^{''}+f^{'}-\frac{f^{'}}{n}\right)|\nabla K|^{2}_{e}\\
&\quad -\phi\frac{nK^{2}f^{'}(K)}{H^{n+3}}|H\nabla_{i}h_{kl}-\nabla_{i}Hh_{kl}|^{2}_{g,h}+\phi H^{n-1}f^{'}(K)\Big|\nabla \left(\frac{K}{H^{n}}\right)\Big|^{2}\\
&\quad  +g\phi^{'}(H)f^{'}K H^{2}+\frac{g\phi^{'}(H)H^{2n}f^{'}(K)}{nK}\Big|\nabla\left( \frac{K}{H^{n}}\right)\Big|^{2}+g\phi^{'}(H)(f-nKf^{'}(K))|A|^{2}\\
&\quad +g\phi^{'}(H)\left(f^{''}+\frac{f^{'}}{K}-\frac{f^{'}}{nK}\right)|\nabla K|^{2}-g\phi^{'}(H)\frac{Kf^{'}(K)}{H^{2}}|H\nabla_{i}h_{kl}-\nabla_{i}Hh_{kl}|^{2}_{g,h}\\
&=Kf^{'}(K)\left[ \Lambda g_{\sigma}-2\phi^{'}\langle \nabla g,\nabla H\rangle_{h}-g\phi^{''}|\nabla H|^{2}_{h} \right] \\
&\quad+\phi f^{'}(K)(1/n-1)\left[ -H^{n}|\nabla g|^{2}_{h}+nH^{-1}K\langle \nabla H, \nabla g\rangle_{h} \right]\\
&\quad +\frac{\phi K f^{'}(K)(n+1)}{H}\Big\langle \nabla H,  \nabla g\Big\rangle_{h} -\phi\frac{nK}{H^{n+1}}(-f+nK f^{'}(K))\left(|A|^{2}-\frac{1}{n}H^{2}\right)\\
&\quad -\phi\frac{1}{H^{n}}\left(Kf^{''}+f^{'}-\frac{f^{'}}{n}\right)|\nabla K|^{2}_{e}\\
&\quad -\phi\frac{nK^{2}f^{'}(K)}{H^{n+3}}|H\nabla_{i}h_{kl}-\nabla_{i}Hh_{kl}|^{2}_{g,h}+\phi H^{n-1}f^{'}(K)\Big|\nabla g\Big|^{2}\\
&\quad  +g\phi^{'}(H)f^{'}K H^{2}+\frac{g\phi^{'}(H)H^{2n}f^{'}(K)}{nK}\Big|\nabla g\Big|^{2}+g\phi^{'}(H)(f-nKf^{'}(K))|A|^{2}\\
&\quad +g\phi^{'}(H)\left(f^{''}+\frac{f^{'}}{K}-\frac{f^{'}}{nK}\right)|\nabla K|^{2}-g\phi^{'}(H)\frac{Kf^{'}(K)}{H^{2}}|H\nabla_{i}h_{kl}-\nabla_{i}Hh_{kl}|^{2}_{g,h}.
\end{split}
\end{equation*}
So, the proof is completed.
\end{proof}
{\bf Proof of Theorem \ref{gxt}:}
We first estimate \eqref{ggg} by using the maximum principal. By means of \eqref{pin}, \eqref{mo1}, \eqref{11}-\eqref{16} into \eqref{ggg}, for $\sigma$ sufficiently small, we get
\begin{equation}
\begin{split}
\label{pp2}
\frac{\partial}{\partial t}g_{\sigma}&\leq Kf^{'}(K)\left[ \Lambda g_{\sigma}-2\phi^{'}\langle \nabla g,\nabla H\rangle_{h}-g\phi^{''}|\nabla H|^{2}_{h} \right] \\
&\quad -\phi f^{'}(K)(1/n-1)H^{n}|\nabla g|^{2}_{h}+2H^{-1}\phi K f^{'}(K)\langle \nabla H, \nabla g\rangle_{h}\\
&\quad +\phi H^{n-1} f^{'}(K)|\nabla g|^{2}-\phi\frac{nK}{H^{n+1}}(-f+nKf^{'}(K))\left( |A|^{2}-\frac{1}{n}H^{2}\right)\\
&\quad -\frac{3\phi C_{1}K f^{'}(K)}{H^{3}}|\nabla H|^{2}+\frac{C_{3}f^{'}(K)\left(\frac{K f^{''}(K)}{f^{'}(K)}+1-\frac{1}{n}\right)}{H^{n+2}}\phi \left(|A|^{2}-\frac{1}{n}H^{2}\right)^{1/2}|\nabla K|^{2}\\
&\quad +g\phi^{'}(H)Kf^{'}(K)\frac{1}{n}\frac{H^{2n}}{K^{2}}\Big|\nabla g\Big|^{2}-g\phi^{'}(H)\frac{Kf^{'}(K)}{H^{2}}|H\nabla_{i}h_{kl}-\nabla_{i}Hh_{kl}|^{2}_{g,h}\\
&\quad +g\phi^{'}(H)(f-nK f^{'}(K))|A|^{2}+ g\phi^{'}(H)f^{'}(K)K H^{2},
\end{split}
\end{equation}
where $C_{1}>0$ and $C_{3}< \infty$ are constants depending only on $M_{0}$.

In view of the choice of $f$, $0<\frac{K f^{''}(K)}{f^{'}(K)}+1-\frac{1}{n}\leq \beta$ showed in Condition (iii), there is
\begin{equation}
\begin{split}
\label{C13}
\frac{C_{3}f^{'}(K)\left(\frac{K f^{''}(K)}{f^{'}(K)}+1-\frac{1}{n}\right)}{H^{n+2}}\phi \left(|A|^{2}-\frac{1}{n}H^{2}\right)^{1/2}|\nabla K|^{2}\leq C_{1}f^{'}(K)K\frac{\phi|\nabla H|^{2}}{H^{3}}.
\end{split}
\end{equation}
Using again  \eqref{pin}, \eqref{mo1}, \eqref{11}-\eqref{16}, \eqref{C13} into \eqref{pp2}, we have
\begin{equation}
\begin{split}
\label{pp3}
\frac{\partial}{\partial t}g_{\sigma}&\leq f^{'}K \Lambda g_{\sigma}-2Kf^{'}\phi^{'}\left[\frac{\langle \nabla g_{\sigma},\nabla H\rangle_{h}}{\phi}- g\frac{\phi^{'}}{\phi}|\nabla H|^{2}_{h}\right]-f^{'}Kg\phi^{''}|\nabla H|^{2}_{h}\\
&\quad -\phi f^{'}(K)(1/n-1)H^{n}\left[\frac{|\nabla g_{\sigma}|^{2}_{h}}{\phi^{2}}-2g\frac{\phi^{'}}{\phi^{2}}\langle \nabla g_{\sigma},\nabla H\rangle_{h}+\frac{g^{2}(\phi^{'})^{2}}{\phi^{2}} |\nabla H|^{2}_{h}\right]\\
&\quad +2 H^{-1}\phi K f^{'}(K)\left[ \frac{\langle \nabla g_{\sigma},\nabla H\rangle_{h}}{\phi}-\frac{g\phi^{'}}{\phi} |\nabla H|^{2}_{h}\right]-2C_{1}f^{'}(K)K\frac{\phi|\nabla H|^{2}}{H^{3}}\\
&\quad +\phi H^{n-1} f^{'}(K)\left[ \frac{|\nabla g_{\sigma}|^{2}}{\phi^{2}}-2 g\frac{\phi^{'}}{\phi^{2}}\langle \nabla g_{\sigma},\nabla H\rangle+\frac{g^{2}(\phi^{'})^{2}}{\phi^{2}}|\nabla H|^{2} \right]\\
&\quad -\phi\frac{nK}{H^{n+1}}(-f+nKf^{'}(K))\left( |A|^{2}-\frac{1}{n}H^{2}\right)\\
&\quad +g\phi^{'}(H)(f-nK f^{'}(K))|A|^{2}+ g\phi^{'}(H)f^{'}(K)K H^{2}\\
&=f^{'}K \Lambda g_{\sigma}-2K f^{'}(K)\frac{\phi^{'}}{\phi}\langle \nabla g_{\sigma},\nabla H\rangle_{h}-f^{'}(K)(1/n-1)H^{n}\frac{1}{\phi}|\nabla g_{\sigma}|^{2}_{h}\\
&\quad+2gf^{'}(K)(1/n-1)H^{n}\frac{\phi^{'}}{\phi}\langle \nabla g_{\sigma}, \nabla H\rangle_{h}+2H^{-1} K f^{'}(K)\langle \nabla g_{\sigma}, \nabla H\rangle_{h}\\
&\quad +H^{n-1}f^{'}(K)\frac{1}{\phi}|\nabla g_{\sigma}|^{2}-2gH^{n-1}f^{'}(K)\frac{\phi^{'}}{\phi}\langle \nabla g_{\sigma}, \nabla H\rangle\\
&\quad+2gKf^{'}\frac{(\phi^{'})^{2}}{\phi}|\nabla H|^{2}_{h}-gKf^{'}(K)\phi^{''}|\nabla H|^{2}_{h}-g^{2}f^{'}(K)H^{n}\frac{(\phi^{'})^{2}}{\phi}(1/n-1)|\nabla H|^{2}_{h}\\
&\quad -2gH^{-1}K f^{'}(K)\phi^{'}|\nabla H|^{2}_{h}+g^{2}H^{n-1}f^{'}(K)\frac{(\phi^{'})^{2}}{\phi}|\nabla H|^{2}-2C_{1}f^{'}(K)K\frac{\phi |\nabla H|^{2}}{H^{3}}\\
&\quad -\phi\frac{nK}{H^{n+1}}(-f+nKf^{'}(K))\left( |A|^{2}-\frac{1}{n}H^{2}\right)\\
&\quad +g\phi^{'}(H)(f-nK f^{'}(K))|A|^{2}+ g\phi^{'}(H)f^{'}(K)K H^{2}.\\
\end{split}
\end{equation}
For better reading, we are in a position to take $\phi (H)=H^{\sigma}$ into \eqref{pp3}, one see
\begin{equation*}
\begin{split}
\label{Up3}
\frac{\partial}{\partial t}g_{\sigma}
&\leq f^{'}K \Lambda g_{\sigma}-2K f^{'}(K)\sigma H^{-1}\langle \nabla g_{\sigma},\nabla H\rangle_{h}-f^{'}(K)(1/n-1)H^{n-\sigma}|\nabla g_{\sigma}|^{2}_{h}\\
&\quad+2\sigma gf^{'}(K)(1/n-1)H^{n-1}\langle \nabla g_{\sigma}, \nabla H\rangle_{h}+2H^{-1} K f^{'}(K)\langle \nabla g_{\sigma}, \nabla H\rangle_{h}\\
&\quad +H^{n-1-\sigma}f^{'}(K)|\nabla g_{\sigma}|^{2}-2\sigma gH^{n-2}f^{'}(K)\langle \nabla g_{\sigma}, \nabla H\rangle\\
&\quad+(\sigma^{2}-\sigma)gKf^{'}(K)H^{\sigma-2}|\nabla H|^{2}_{h}-g^{2}\sigma^{2}f^{'}(K)H^{n} H^{\sigma-2}(1/n-1)|\nabla H|^{2}_{h}\\
&\quad +\sigma^{2}g^{2}H^{n-1}f^{'}(K) H^{\sigma-2}|\nabla H|^{2}-2C_{1}f^{'}(K)K\frac{H^{\sigma} |\nabla H|^{2}}{H^{3}}\\
&\quad -H^{\sigma}\frac{nK}{H^{n+1}}(-f+nKf^{'}(K))\left( |A|^{2}-\frac{1}{n}H^{2}\right)\\
&\quad +\sigma g H^{\sigma-1}f^{'}(K)K(H^{2}-n |A|^{2})+\sigma g H^{\sigma-1}f |A|^{2}.
\end{split}
\end{equation*}
By virtue of \eqref{pin} and Lemma \ref{AH}, there is
\begin{equation*}
\begin{split}
\label{Up3}
&H^{\sigma}\frac{nK}{H^{n+1}}(-f+nKf^{'}(K))\left( |A|^{2}-\frac{1}{n}H^{2}\right)\\
&\geq C \delta Hg_{\sigma}(-f+nKf^{'}(K)),
\end{split}
\end{equation*}
this together with $|A|^{2}\leq |H|^{2}$ to embrace
\begin{equation*}
\begin{split}
\label{Up3}
&-H^{\sigma}\frac{nK}{H^{n+1}}(-f(K)+nKf^{'}(K))\left( |A|^{2}-\frac{1}{n}H^{2}\right)+\sigma g H^{\sigma-1}f |A|^{2}\\
&\leq -C \delta Hg_{\sigma}(-f(K)+nKf^{'}(K))+\sigma H g_{\sigma} f(K).
\end{split}
\end{equation*}
In view of the choice of $f$, $\alpha_{1}f(K)\leq nK f^{'}(K)-f(K)$ showed in Condition (ii), this together with corollary \ref{KPR},  for $\sigma$ sufficiently small, we have
\begin{equation}
\begin{split}
\label{Up3}
&-H^{\sigma}\frac{nK}{H^{n+1}}(-f(K)+nKf^{'}(K))\left( |A|^{2}-\frac{1}{n}H^{2}\right)+\sigma g H^{\sigma-1}f(K)|A|^{2}\leq 0.
\end{split}
\end{equation}
This reveals that
\begin{equation}
\begin{split}
\label{FP}
\frac{\partial}{\partial t}g_{\sigma}&
\leq f^{'}K \Lambda g_{\sigma}-2K f^{'}(K)\sigma H^{-1}\langle \nabla g_{\sigma},\nabla H\rangle_{h}-f^{'}(K)(1/n-1)H^{n-\sigma}|\nabla g_{\sigma}|^{2}_{h}\\
&\quad+2gf^{'}(K)(1/n-1)H^{n-1}\langle \nabla g_{\sigma}, \nabla H\rangle_{h}+2H^{-1} K f^{'}(K)\langle \nabla g_{\sigma}, \nabla H\rangle_{h}\\
&\quad +H^{n-1-\sigma}f^{'}(K)|\nabla g_{\sigma}|^{2}-2gH^{n-2}f^{'}(K)\langle \nabla g_{\sigma}, \nabla H\rangle_{h}-C_{1}Kf^{'}(K)H^{\sigma-3}|\nabla H|^{2},
\end{split}
\end{equation}
provided $h_{ij}\geq \varepsilon H g_{ij}$ at $t=0$. Then, applying the parabolic maximum principle to \eqref{FP}, there is
\begin{equation*}\label{}
\frac{\partial}{\partial t}g_{\sigma}\leq 0,
\end{equation*}
which implies the desired result.

Now, based on the previous result, we shall show $C^{\infty}$-convergence to a strictly convex hypersurface $M_{\infty}$ in the spirit of Chou-Wang\cite{CW00}.

We first parametrize $M_{t}$ by the inverse of the Gauss map since $M_{t}$ is strictly convex. This implies that $F:\sn \rightarrow M_{t}$ takes a unit vector $x\in \sn$ to the point $F(x,t)$ on $M_{t}$ having $x$ as its outward normal. The support function $h$ of $M_{t}$ is given by $u(x,t)=\langle F(x,t),x\rangle$. So, \eqref{ca} is equivalent to
\begin{equation}\label{}
\frac{\partial }{\partial t}h(x,t)=-f(K(x,t)),\quad x\in \sn.
\end{equation}
In the subsequence, we recall some facts. Let $e_{ij}$ be the standard metric of the sphere $\sn$. Differentiating  $u(x,t)=\langle F(x,t),x\rangle$, we get
\[
\nabla_{i}u=\langle \nabla_{i}x, F(x,t)\rangle+\langle x, \nabla_{i} F(x,t)\rangle.
\]
Since $\nabla_{i}F(x,t)$ is tangent to $M_{t}$ at $F(x,t)$, one see
\[
\nabla_{i}u=\langle \nabla_{i}x, F(x,t)\rangle.
\]
This implies
\begin{equation*}\label{}
F(x,t)=\nabla_{\sn}u+u x.
\end{equation*}
The second fundamenatal form $h_{ij}$ of $M_{t}$ is given by
\begin{equation}\label{uI}
h_{ij}=\nabla_{ij}u+u e_{ij},
\end{equation}
and the induced metric matrix $g_{ij}$ of $M_{t}$ can be obtained by Weingarten's formula,
\begin{equation}\label{gI}
e_{ij}=\langle \nabla_{i}x, \nabla_{j}x\rangle = h_{ik}h_{lj}g^{kl}.
\end{equation}
The radii of principal curvature are the eigenvalues of the matrix $w_{ij}=h^{ik}g_{jk}$. When considering a smooth local orthonormal frame on $\sn$, using \eqref{uI} and \eqref{gI}, there is
\[
w_{ij}=\nabla_{ij}u+u\delta_{ij}.
\]
The Gauss curvature of $F(x,t)\in M_{t}$ is given by
\[
K(x)=\frac{1}{\det(u_{ij}+u\delta_{ij})}.
\]

In \cite{CW00}, Chou-Wang concerned with a logarithmic version of the Gauss curvature flow for convex hypersurfaces. Let $R(t)$ and $r(t)$ be the outer and inner radii of the hypersurface $\partial \mathcal{N}_{t}$ determined by $h(x,t)$ respectively. Setting
\[
R_{0}=\sup \{R(t):t\in (0,T)\}
\]
and
\[
r_{0}=\inf \{r(t):t\in (0,T)\}.
\]
They have estimated the principal of curvature of $\partial \mathcal{N}_{t}$ from both side in terms of $r^{-1}_{0}, R_{0}$ and initial data as follows.
\begin{lem}\cite[Lemm 2.2]{CW00}\label{cw}
Let $r(t)$ and $R(t)$ be the inner and outer radii of a strictly convex hypersurface $\mathcal{N}_{t}$ respectively. There exists a dimensional constant $C$ such that
\[
\frac{R^{2}(t)}{r(t)}\leq C\max_{y\in \partial \mathcal{N}_{t}}\chi_{y; \mathcal{N}_{t}},
\]
where $\chi_{y;\partial \mathcal{N}_{t}}$ is the maximal principal radius of $\mathcal{N}_{t}$ at the point $y$.
\end{lem}
Motivated by Chou-Wang's work,  we shall give an upper estimate for the principal radii of curvature. Here we have assumed that the origin is translated so that $u>0$ on $\sn$.
\begin{lem}\label{PL*}
  Under the flow \eqref{ca}.  Then there exists a positive constant $C$, independent of $t$, such that
\begin{equation}\label{PUL}
\chi (\{w_{ij}\})\leq C(1+R_{0}),
\end{equation}
where $w_{ij}=u_{ij}+u\delta_{ij}$, and $\lambda (\{w_{ij}\})$ are the eigenvalues of $\{w_{ij}\}$.
\end{lem}

\begin{proof}
To obtain \eqref{PUL}, we are in a position to set the auxiliary function as
\begin{equation}\label{LFun}
\psi(x,t)=\chi_{max}(\{w_{ij}\})+|\nabla_{\sn} u|^{2}+u^{2}.
\end{equation}
Here $\lambda_{max}(\{w_{ij}\})$ is the maximal eigenvalue of $\{w_{ij}\}$.

Now we shall assume that the maximum of $\psi(x,t)$ is achieved at $(x_{0},t_{0})$ on $\sn\times (0,T)$. By virtue of a rotation of coordinates, we may suppose that $\{w_{ij}(x_{0},t_{0})\}$ is diagonal, and $\chi_{max}(\{w_{ij}\})(x_{0},t_{0})=w_{11}(x_{0},t_{0})$. So, $\eqref{LFun}$ becomes the following form:
\begin{equation}\label{LFun2}
\widetilde{\psi}(x,t)=w_{11}+|\nabla_{\sn} u|^{2}+u^{2}.
\end{equation}
Since  $\widetilde{\psi}(x,t)$ attains the maximum at $(x_{0},t_{0})$, it follows that, at $(x_{0},t_{0})$, there is
\begin{equation}
\begin{split}
\label{gaslw1}
0=\nabla_{i}\widetilde{\psi}&=\nabla_{i}w_{11}+2\sum_{j} u_{j}u_{ji}+2uu_{i}\\
&=(u_{i11}+u_{1}\delta_{1i})+2u_{i}u_{ii}+2uu_{i},
\end{split}
\end{equation}
and we also have
\begin{align}\label{gaslw2}
0&\geq \nabla_{ii}\widetilde{\psi}=\nabla_{ii}w_{11}
+2\left(\sum_{j} u_{j}u_{jii}+u^{2}_{ii}\right)+2u^{2}_{i}+2uu_{ii}.
\end{align}
Furthermore,
\begin{equation}
\label{gaslw3}
0\leq\partial_{t}\widetilde{\psi}=\partial_{t}w_{11}+2\sum_{j} u_{j}u_{jt}+2uu_{t}=(u_{11t}+u_{t})+2\sum_{j} u_{j}u_{jt}+2uu_{t}.
\end{equation}
Due to
\begin{equation}\label{gaslw4}
u_{t}=-f(K).
\end{equation}
Taking the covariant derivative of \eqref{gaslw4} with respect to $e_{j}$, yields
\begin{equation}
\begin{split}
\label{gaslw6}
u_{tj}&=f^{'}(K)K^{2}\sum_{i,k} \sigma_{n}^{ik}\nabla_{j}w_{ik}\\
&=f^{'}(K)K^{2}\sum_{i} \sigma_{n}^{ii}(u_{jii}+u_{i}\delta_{ij}),
\end{split}
\end{equation}
where $\sigma_{n}$ is the $n$-th symmetric function for the radii of principal curvature, i.e., $\sigma_{n}=K^{-1}$, and $\sigma^{ik}_{n}=\frac{\partial \sigma_{n}}{\partial w_{ik}}$. Differentiating \eqref{gaslw6}, we have
\begin{equation}
\begin{split}
\label{gaslw7}
&u_{11t}=f^{'}(K)K^{2}\sum_{i,k}\sigma^{ik}_{n}\nabla_{11}w_{ik}-\sum_{i,k,\beta,s}\frac{\partial^{2}f(K)}{\partial w_{ik}\partial w_{\beta s}}\nabla_{1}w_{ik}\nabla_{1}w_{\beta s}\\
&=-\sum_{i}\frac{\partial f(K)}{\partial W_{ii}}\nabla_{11}w_{ii}-\sum_{i,k,\beta,s}\frac{\partial^{2}f(K)}{\partial w_{ik}\partial w_{\beta s}}\nabla_{1}w_{ik}\nabla_{1}w_{\beta s},
\end{split}
\end{equation}
where $\sigma_{n}^{ik,\beta s}=\frac{\partial^{2}\sigma_{n}}{\partial w_{ik}\partial w_{\beta s}}$. Recall that the Ricci identity on sphere reads
\begin{equation*}
\nabla_{11}w_{ij}=\nabla_{ij}w_{11}-\delta_{ij}w_{11}+\delta_{11}w_{ij}-\delta_{1i}w_{1j}+\delta_{1j}w_{1i}.
\end{equation*}
So, with the aid of the Ricci identity, \eqref{gaslw2}, \eqref{gaslw3}, \eqref{gaslw6}, \eqref{gaslw7} and the convexity of $f(K)$, we obtain
\begin{equation}
\begin{split}
\label{gaslw8}
0&\geq -\sum_{i}\frac{\partial f(K)}{\partial w_{ii}} \nabla_{ii}\widetilde{\psi}-\partial_{t}\widetilde{\psi}\\
&=-\sum_{i}\frac{\partial f(K)}{\partial w_{ii}}\nabla_{ii}w_{11}-2\sum_{i}\frac{\partial f(K)}{\partial w_{ii}}\left(\sum_{j}u_{j}u_{jii}+u^{2}_{ii}\right)-2\sum_{i}\frac{\partial f(K)}{\partial w_{ii}}u^{2}_{i}-2\sum_{i}\frac{\partial f(K)}{\partial w_{ii}}uu_{ii}\\
&\quad -w_{11t}-2\sum_{j}u_{j}u_{jt}-2uu_{t}\\
&=-\sum_{i}\frac{\partial f(K)}{\partial w_{ii}}(\nabla_{11}w_{ii}+w_{11}-w_{ii})-2\sum_{j}u_{j}(\sum_{i}\frac{\partial f(K)}{\partial w_{ii}}u_{jii}-u_{jt})-2\sum_{i}\frac{\partial f(K)}{\partial w_{ii}}u^{2}_{ii}\\
&\quad -2\sum_{i}\frac{\partial f(K)}{\partial w_{ii}}u^{2}_{i}-2\sum_{i}\frac{\partial f(K)}{\partial w_{ii}}u(w_{ii}-u)-u_{11t}-u_{t}-2uu_{t}\\
&\geq-\sum_{i}\frac{\partial f(K)}{\partial w_{ii}}\nabla_{11}w_{ii}-u_{11t}+2\sum_{j}u_{j}\frac{\partial f(K)}{\partial w_{jj}}u_{j}-2\sum_{i}\frac{\partial f(K)}{\partial w_{ii}}(w_{ii}-u)^{2}\\
&\quad -2\sum_{i}\frac{\partial f(K)}{\partial w_{ii}}u^{2}_{i}-2\sum_{i}\frac{\partial f(K)}{\partial w_{ii}}u w_{ii}+2\sum_{i}\frac{\partial f(K)}{\partial w_{ii}}u^{2}-(2u+1)u_{t}\\
&\geq -2\sum_{i}\frac{\partial f(K)}{\partial w_{ii}}w^{2}_{ii}+2\sum_{i}\frac{\partial f(K)}{\partial w_{ii}}w_{ii}u-(2u+1)u_{t}\\
&=2f^{'}(K)K^{2}\sum_{i}\sigma^{ii}_{n}w^{2}_{ii}-2u f^{'}(K)K^{2} \sum_{i}\sigma^{ii}_{n} w_{ii}+(2u+1)f(K).
\end{split}
\end{equation}
This implies that
\begin{equation}\label{Fmax}
2f^{'}(K)K^{2}\sigma_{n}^{11}w_{11}^{2}\leq 2nuf^{'}(K)K-(2u+1)f(K).
\end{equation}
 Since $\{w_{ij}\}$ is diagonal at $(x_{0},t_{0})$, and $w_{ii}=\chi_{i}$, $\forall i=1,\ldots, n$. Then we obtain
\begin{equation}\label{F11}
\sigma_{n}^{11}=\sigma_{n-1}(\chi|1),
\end{equation}
where $\lambda:=(\chi_{1},\ldots,\chi_{n})$, $\sigma_{n-1}(\chi|1)$ denotes the $(k-1)$-th symmetric functions with $\lambda_{1}=0$. By using Newton-MacLaurin inequality\cite{CW01}, one has
\begin{equation}\label{Mac}
\left[\frac{\sigma_{n-1}(\chi|1)}{C^{n-1}_{n}}\right]^\frac{1}{n-1}\geq \left[\frac{\sigma_{n}(\chi|1)}{C^{n}_{n}}\right]^\frac{1}{n},
\end{equation}
using \eqref{Mac}, for some dimensional positive constants $C$, we get
\begin{equation}
\begin{split}
\label{Fw11}
\sigma_{n}(\chi|1)\leq C \sigma_{n-1}(\chi|1)^\frac{n}{n-1}\leq C\chi_{1}\sigma_{n-1}(\chi|1).
\end{split}
\end{equation}
Substituting \eqref{Fw11} into $\sigma_{n}(\chi)=\sigma_{n}(\chi|1)+\chi_{1}\sigma_{n-1}(\chi|1)$, we have
\begin{equation}\label{1cq}
\chi_{1}\sigma_{n-1}(\chi|1)\geq C \sigma_{n}(\chi).
\end{equation}
So, employing \eqref{1cq}, one see that
\begin{equation}
\begin{split}
\label{maxw11}
\frac{\sigma_{n}^{11}w^{2}_{11}}{\sigma_{n}}&=\frac{\sigma_{n}(\chi|1)\chi^{2}_{1}}{\sigma_{n}}\geq \frac{C\sigma_{n}\chi_{1}}{\sigma_{n}}=Cw_{11}.
\end{split}
\end{equation}
Hence, substituting \eqref{maxw11} into \eqref{Fmax}, there is
\begin{equation}\label{CW11}
2Cf^{'}(K) Kw_{11}\leq 2 nu f^{'}(K) K- (2u+1)f(K).
\end{equation}
In the light of the choice of $f$, $Kf^{'}(K)\leq  f(K)$ as showed in Condition (ii), we get
\begin{equation}
\begin{split}
\label{Up3}
w_{11}&\leq C\left[2  nu - (2u+1)\frac{f(K)}{Kf^{'}(K)}\right]\leq C[(2n-2)u-1]\\
&\leq C[2(n-1)R_{0}-1]\leq \widetilde{C}( R_{0}+1).
\end{split}
\end{equation}
Therefore, the proof is completed.
\end{proof}

Next, we are ready to obtain the upper bound of $K$.

\begin{lem}\label{prin}
  Under the flow \eqref{ca}. Then there exists a positive constant $\bar{C}$, independent of $t$, such that
\[
K\leq C.
\]

\end{lem}
\begin{proof}
Set
\begin{equation}\label{}
\zeta(t)=\frac{1}{\omega_{n}}\int_{\sn}xu(x,t)dx
\end{equation}
be the Steiner point of $\Omega_{t}$. By means of Lemma \ref{cw},  there exists a positive $\varepsilon_{0}$, which depends only on $n$, $r_{0}$ and $R_{0}$, independent of $t$,  such that
\[
u(x,t)-\zeta(t)\cdot x\geq 2\varepsilon_{0}.
\]
Now establishing the auxiliary function as
\begin{equation}\label{AF1}
\Re(x,t)=\frac{f(K)}{u-\zeta(t)\cdot x-\varepsilon_{0}}=\frac{-u_{t}}{u-\zeta(t)\cdot x-\varepsilon_{0}}.
\end{equation}
For any fixed $t\in(0,+\infty)$, suppose that the (positive) maximum of $Q(x,t)$ is achieved at $x_{0}$. Thus, we get that at $x_{0}$,
\begin{equation}\label{Up1}
0=\nabla_{i}\Re=\frac{-u_{ti}}{u-\zeta_{0}-\varepsilon_{0}}+\frac{u_{t}(u_{i}-\zeta_{i})}{(u-\zeta_{0}-\varepsilon_{0})^{2}},
\end{equation}
where $\zeta_{i}:=\zeta\cdot e_{i}$, and $\zeta_{0}=\zeta\cdot x_{0}$.

Then, applying \eqref{Up1}, at $x_{0}$, we also have
\begin{equation}
\begin{split}
\label{Up2}
0\geq\nabla_{ij}\Re&=\frac{-u_{tij}}{u-\zeta_{0}-\varepsilon_{0}}+\frac{u_{ti}(u_{j}-\zeta_{j})+u_{tj}(u_{i}-\zeta_{i})+u_{t}(u_{ij}+\zeta_{0}\delta_{ij})}{(u-\zeta_{0}-\varepsilon_{0})^{2}}-\frac{2u_{t}(u_{i}-\zeta_{i})(u_{j}-\zeta_{j})}{(u-\zeta_{0}-\varepsilon_{0})^{3}}\\
&=\frac{-u_{tij}}{u-\zeta_{0}-\varepsilon_{0}}+\frac{u_{t}u_{ij}}{(u-\zeta_{0}-\varepsilon_{0})^{2}}+\frac{u_{t}\zeta_{0}\delta_{ij}}{(u-\zeta_{0}-\varepsilon_{0})^{2}}.
\end{split}
\end{equation}
 \eqref{Up2} implies
\begin{equation*}
\begin{split}
\label{Up3}
-u_{tij}-u_{t}\delta_{ij}&\leq - \frac{u_{t}u_{ij}}{u-\zeta_{0}-\varepsilon_{0}}-\frac{u_{t}\zeta_{0}\delta_{ij}}{(u-\zeta_{0}-\varepsilon_{0})}-u_{t}\delta_{ij}\\
&=\frac{-u_{t}}{u-\zeta_{0}-\varepsilon_{0}}[u_{ij}+(u-\zeta_{0}-\varepsilon_{0})\delta_{ij})]+\Re\zeta_{0}\delta_{ij}\\
&=\Re(w_{ij}-\varepsilon_{0}\delta_{ij}-\zeta_{0}\delta_{ij})+\Re\zeta_{0}\delta_{ij}\\
&=\Re(w_{ij}-\varepsilon_{0}\delta_{ij}).
\end{split}
\end{equation*}
On the other hand,  at $x_{0}$, we have
\begin{equation}
\begin{split}
\label{s2k} \partial_{t}\Re&=\frac{-u_{tt}}{u-\zeta_{0}-\varepsilon_{0}}+\frac{u^{2}_{t}}{(u-\zeta_{0}-\varepsilon_{0})^{2}}-\frac{u_{t}\frac{d\zeta_{0}}{dt}}{(u-\zeta_{0}-\varepsilon_{0})^{2}}\\
&=\frac{1}{u-\zeta_{0}-\varepsilon_{0}}\left(\frac{ \partial f(K)}{\partial t}\right)+\Re^{2}-\frac{u_{t}\frac{d\zeta_{0}}{dt}}{(u-\zeta_{0}-\varepsilon_{0})^{2}}.
\end{split}
\end{equation}
Rotate the axes so that $\{w_{ij}\}$ is diagonal at $x_{0}$, we have
\begin{equation}
\begin{split}
\label{sigmak41}
\frac{\partial f(K)}{\partial t}&=-f^{'}(K)\sigma^{-2}_{n}\sum_{i,j}\frac{\partial \sigma_{n}}{\partial w_{ij}}(u_{tij}+u_{t}\delta_{ij})\\
&\leq f^{'}(K)\sigma^{-2}_{n}\sum_{i,j}\frac{\partial \sigma_{n}}{\partial w_{ij}}(w_{ij}-\varepsilon_{0}\delta_{ij})\Re\\
&=f^{'}(K)\sigma^{-2}_{n}\left(n\sigma_{n}-\varepsilon_{0}\sum_{i} \sigma^{ii}_{n}\right)\Re\\
&\leq f^{'}(K)(n\sigma^{-1}_{n}-C\varepsilon_{0}\sigma^{-1-\frac{1}{n}}_{n})\Re.
\end{split}
\end{equation}
The last inequality we use Newton-MacLaurin inequality $\sum_{i}\sigma^{ii}_{n}\geq C\sigma^{1-\frac{1}{n}}_{n}$, where $C$ is a positive constant depending only on $n$ and $k$. In the light of the choice of $f$, $\frac{1}{n}f(K)< Kf^{'}(K)\leq f(K)$, $K\geq \hat{\gamma} f(K)^{\gamma}$ for $K\gg 1$, showed in Conditions (ii) and (iii). This together with \eqref{sigmak41} to substitute them into \eqref{s2k}, providing $Q\gg 1$, we get
\begin{equation}
\begin{split}
\label{ODE}
\partial_{t}\Re&\leq \frac{nf^{'}(K)K}{u-\zeta_{0}-\varepsilon_{0}}\Re-\frac{C\varepsilon_{0}K f^{'}(K)}{u-\zeta_{0}-\varepsilon_{0}}K^{\frac{1}{n}}\Re+\Re^{2}+\frac {|u_{t}||\frac{d \zeta_{0}}{dt}|}{(u-\zeta_{0}-\varepsilon_{0})^{2}}\\
&\leq -C_{2}\Re^{2+\frac{\gamma}{n}}+C_{3}\Re^{2}<0
\end{split}
\end{equation}
 for some positive $C_{2},C_{3}$, independent of $t$. Therefore, \eqref{ODE} implies that
\[
\Re(x_{0},t)\leq C
\]
for $C>0$, independent of $t$. This implies that
\[
K\leq \hat{C}
\]
for $\hat{C}>0$, independent of $t$. The proof is completed.
\end{proof}

On the other hand, by means of Corollary \ref{KPR}, and the choice of $f$, we know that $-\partial h/\partial t \geq \inf_{t=0}f(K)>0$, hence $h(x,t)\leq \max_{t=0}h(x,t)$. Observe that the solution of \eqref{ca} will shrink to a point if $\vartheta$ is small enough, then we put
\[
\vartheta_{*}=\max \{\vartheta>0: F(\cdot, t) {\rm \ shrinks  \ to \ a \ point \ in \ finite \ time } \}.
\]
Combining Lemmas \ref{cw}, \ref{PL*}, \ref{prin}, we conclude that, for any $\vartheta\in (\vartheta_{*},+\infty)$, the inner radii of $F(\cdot,t)$ have a uniform positive lower bound and the outer radii are uniformly bounded from above,  so \eqref{ca} is uniformly parabolic on any finite time and there exists a positive constant $C$, independent of $t$, such that
\begin{equation}\label{2a}
||u||_{C^{2}(\sn \times [0,T))}\leq C.
\end{equation}
The $C^{2,\alpha}$ estimate of $u$ follows from the standard parabolic theory showed in the result of \cite{K87}, and the $C^{\infty}$ estimate of $u$ follows from the standard arguments. Hence, flow \eqref{ca} has a long time solution and we reveal that a subsequence of $M_{t}$ converges in $C^{\infty}$ to a hypersurface $M_{\infty}$.

In the last, we show that the estimate in Theorem \ref{conv} implies that $M_{\infty}$ is a round sphere.  Precisely,
\begin{equation}\label{FR}
0\leq \frac{1}{n^{n}}-\frac{ K}{ H^{n}}=\frac{g_{\sigma}}{H^{\sigma}}\leq \frac{\max g_{\sigma}(x,0)}{H^{\sigma}}\rightarrow 0 \ as \ t\rightarrow \infty
\end{equation}
for some points where the mean curvature is large in a neighborhood $U$ of $M_{\infty}$. By scale the invariance of $g$, $g=0$ in $U$. Clearly, at all points of $M_{\infty}$, we have $g=0$, i.e.,
\[
\frac{1}{n^{n}}=\frac{K}{H^{n}},
\]
which tells that $M_{\infty}$ is  a round sphere. So, we complete the proof of Theorem \ref{th*}.

{\bf Example: } we are of great interest for the following function:
\begin{equation}\label{sf}
f(K)=K^\frac{1}{n}\ln \hat{K},
\end{equation}
where $\hat{K}=K+ K_{0}$ with $K_{0}\geq e^{n/(n-1)}$ a constant chosen.  Clearly,
\begin{equation*}\label{c1}
f^{'}(K)=\frac{1}{n}K^{\frac{1}{n}-1}\ln \hat{K}+K^{\frac{1}{n}}\frac{1}{\hat{K}},
\end{equation*}

\begin{equation*}\label{c2}
f^{''}(K)=\frac{1}{n}(\frac{1}{n}-1)K^{\frac{1}{n}-2}\ln \hat{K}+\frac{2}{n}K^{\frac{1}{n}-1}\frac{1}{\hat{K}}-K^{\frac{1}{n}}\frac{1}{(\hat{K})^{2}},
\end{equation*}
\begin{equation*}
\begin{split}
\label{c3}
Kf^{'}(K)
&= K^{\frac{1}{n}} \left(\frac{1}{n} \ln \hat{K} + \frac{K}{\hat{K}}\right)\\
&\leq K^{\frac{1}{n}}\left(\frac{1}{n} \ln \hat{K} + 1\right)\\
&  \leq K^{\frac{1}{n}} \ln \hat{K}=f(K),
\end{split}
\end{equation*}
\begin{equation*}
\begin{split}
\label{c4}
\frac{1}{n}-1<\frac{Kf^{''}(K)}{f^{'}(K)}=\frac{\frac{1}{n}(\frac{1}{n}-1)K^{\frac{1}{n}-1}\ln \hat{K}+\frac{2}{n}K^{\frac{1}{n}}\frac{1}{\hat{K}}-K^{\frac{1}{n}+1}\frac{1}{(\hat{K})^{2}}}{\frac{1}{n}K^{\frac{1}{n}-1}\ln \hat{K}+K^{\frac{1}{n}}\frac{1}{\hat{K}}}\leq \frac{2}{n},
\end{split}
\end{equation*}
and
\begin{equation*}
\begin{split}
\label{c5}
n K f^{'}(K)&= K^{\frac{1}{n}}\ln \hat{K} + K^{\frac{1}{n}}\frac{n K}{\hat{K}}\\
&= f(K)\left( 1+\frac{ nK}{\hat{K}\ln \hat{K}}\right)> f(K).
\end{split}
\end{equation*}
Observe that \eqref{sf} satisfies Condition 1.1. Here $\beta$ in (ii) can be chosen as $(n+1)/n$, $\gamma$ and $\hat{\gamma}$ in (iv) can be respectively chosen as $n/(n+1)$ and $1/(2^{\frac{n}{n+1}})$. Combining the above, we see that the $C^{\infty}$ estimates of $u$ holds. In conjunction with Theorem \ref{AH}, we conclude that Theorem \ref{th*} holds. Note that \eqref{sf} is a non-homogeneous function of the principal curvature, unlike many previously studied $\alpha$-Gauss curvature flows.


\begin{thebibliography}{20}

\bibitem{An94}
B. Andrews, Contraction of convex hypersurfaces in Euclidean space, Calc. Var. Partial Differential Equations {\bf 2} (1994), no.~2, 151--171.

\bibitem{An96}
B. Andrews, Contraction of convex hypersurfaces by their affine normal, J. Differential Geom. {\bf 43} (1996), no.~2, 207--230.

\bibitem{An07}
B. Andrews, Pinching estimates and motion of hypersurfaces by curvature functions, J. Reine Angew. Math. {\bf 608} (2007), 17--33.

\bibitem{An10}
B. Andrews, Moving surfaces by non-concave curvature functions, Calc. Var. Partial Differential Equations {\bf 39} (2010), no.~3-4, 649--657.


\bibitem{Ag16}
B. Andrews, P. Guan\ and\ L. Ni, Flow by powers of the Gauss curvature, Adv. Math. {\bf 299} (2016), 174--201.


\bibitem{BCD17}
S. Brendle, K. Choi\ and\ P. Daskalopoulos, Asymptotic behavior of flows by powers of the Gaussian curvature, Acta Math. {\bf 219} (2017), no.~1, 1--16.


\bibitem{CW85}
B. Chow, Deforming convex hypersurfaces by the $n$th root of the Gaussian curvature, J. Differential Geom. {\bf 22} (1985), no.~1, 117--138.

\bibitem{CW98}
B. Chow\ and\ D.-H. Tsai, Nonhomogeneous Gauss curvature flows, Indiana Univ. Math. J. {\bf 47} (1998), no.~3, 965--994.

\bibitem{CW00}
K.-S. Chou\ and\ X.-J. Wang, A logarithmic Gauss curvature flow and the Minkowski problem, Ann. Inst. H. Poincar\'{e} C Anal. Non Lin\'{e}aire {\bf 17} (2000), no.~6, 733--751.

\bibitem{CW01}
K.-S. Chou\ and\ X.-J. Wang, A variational theory of the Hessian equation, Comm. Pure Appl. Math. {\bf 54} (2001), no.~9, 1029--1064.



\bibitem{Es22}
T. Espin, A pinching estimate for convex hypersurfaces evolving under a non-homogeneous variant of mean curvature flow, Proc. Edinb. Math. Soc. (2) {\bf 65} (2022), no.~2, 376--391.


\bibitem{F74}
W. J. Firey, Shapes of worn stones, Mathematika {\bf 21} (1974), 1--11.

\bibitem{Ge06}
C. Gerhardt, {\it Curvature problems}, Series in Geometry and Topology, 39, International Press, Somerville, MA, 2006.

\bibitem{H84}
G. Huisken, Flow by mean curvature of convex surfaces into spheres, J. Differential Geom. {\bf 20} (1984), no.~1, 237--266.

\bibitem{K87}
N. V. Krylov, {\it Nonlinear elliptic and parabolic equations of the second order}, translated from the Russian by P. L. Buzytsky [P. L. Buzytski\u{\i}], Mathematics and its Applications (Soviet Series), 7, D. Reidel Publishing Co., Dordrecht, 1987.

\bibitem{S06}
F. Schulze, Convexity estimates for flows by powers of the mean curvature, Ann. Sc. Norm. Super. Pisa Cl. Sci. (5) {\bf 5} (2006), no.~2, 261--277.

\bibitem{T85}
K. Tso, Deforming a hypersurface by its Gauss-Kronecker curvature, Comm. Pure Appl. Math. {\bf 38} (1985), no.~6, 867--882.





\end{thebibliography}
\end{document}